\documentclass{amsart}
\usepackage{amsfonts, graphicx, amsmath, amssymb, color, verbatim}
\usepackage{pdfsync}
\usepackage[all, 2cell,dvips]{xypic}
\usepackage[margin = 1.5in]{geometry}

\newcommand{\p}{\partial}

\newcommand \JJ {\mathcal{J}}

\newcommand \absv [1]{\left \lvert #1 \right \rvert }
\newcommand \lp {\left(}
\newcommand \rp {\right)}
\newcommand{\set}[1]{\left\{ #1 \right\} }
\newcommand{\wt}[1]{\widetilde{#1}}
\newcommand{\norm}[2][]{\left \| #2 \right \|_{#1} }

\newcommand\RR{\mathbb{R}}
\newcommand\Id{\operatorname{Id}}
\newcommand\dOmega{\partial \Omega}
\newcommand\inc{\textrm{in}}
\newcommand\out{\textrm{out}}

\newtheorem{theorem}{Theorem}

\newtheorem{lemma}[theorem]{Lemma}

\newtheorem{corollary}[theorem]{Corollary}
\DeclareMathOperator \spn {span}
\DeclareMathOperator \supp {supp}

\title{Potential scattering and the continuity of phase-shifts}
\author{Jesse Gell-Redman and Andrew Hassell}

\begin{document}

\maketitle

\begin{abstract}
Let $S(k)$ be the scattering matrix for a Schr\"odinger operator (Laplacian plus potential) on $\RR^n$ with compactly supported smooth potential. It is well known that $S(k)$ is unitary and that the spectrum of $S(k)$ accumulates on the unit circle only at $1$; moreover, $S(k)$ depends analytically on $k$ and therefore its eigenvalues depend analytically on $k$ provided they stay away from $1$. 

We give examples of smooth, compactly supported potentials on $\RR^n$ for which (i) the scattering matrix $S(k)$ does not have $1$ as an eigenvalue for any $k > 0$,  and (ii)  there exists $k_0 > 0$ such that there is an analytic eigenvalue branch $e^{2i\delta(k)}$ of $S(k)$ converging to $1$ as $k \downarrow k_0$. 
This shows that the eigenvalues of the scattering matrix, as a function of $k$, do not necessarily have continuous extensions to or across the value $1$. In particular this shows that a `micro-Levinson theorem' for non-central potentials in $\RR^3$  claimed in a 1989 paper of R. Newton is incorrect. 
\end{abstract}

\section{Introduction}

In this article, we consider scattering in $\RR^n$ due to a nonpositive potential function, which we call a potential well. We denote the potential $-V(x)$, $V \geq 0$, and assume for simplicity that $V$ is smooth and compactly supported. Recall that the Schr\"odinger operator $H = \Delta - V$, where $\Delta = - \sum \partial_{x_i}^2$, has absolutely continuous spectrum on $(0, \infty)$ and may have finitely many eigenvalues on the nonpositive real axis. The scattering matrix, $S(k)$, $k > 0$, can be defined in terms of the generalized eigenfunctions or scattering solutions for $H$. As is well known, for each smooth function $q_\inc$ on the sphere, there is a unique solution $u$ to $(H - k^2)u = 0$, with $u$ taking the form
$$
u = r^{-(n-1)/2} \Big( e^{-ikr} q_\inc(\omega) + e^{ikr} q_\out(-\omega) \Big) + O(r^{-(n+1)/2}),
$$
as $r = |x| \to \infty$ \cite{GST}. Here $q_\out \in C^\infty(S^{n-1})$. As a consequence of uniqueness,  $q_\out$ is determined by $q_\inc$; the map $q_\inc \mapsto e^{i\pi(n-1)/2} q_\out$  is by definition the scattering matrix $S(k)$. The normalization factor $e^{i\pi(n-1)/2}$ is chosen so that this `stationary' definition of the scattering matrix agrees with time-dependent definitions (see e.g.\ \cite{RSIII} or \cite{Yafaev}), and is such that the scattering matrix for the zero potential is the identity. The scattering matrix $S(k)$ extends to a unitary map from $L^2(S^{n-1})$ to $L^2(S^{n-1})$ for each $k > 0$. 

In this note, we are interested in the eigenvalues of the scattering matrix $S(k)$. As $S(k)$ is unitary, these lie on the unit circle, and they are conventionally denoted $e^{2i\delta_j(k)}$ where $\delta_j$ is real; the $\delta_j$ are called phase shifts. They are determined up to a multiple of $\pi$. 

The scattering matrix $S(k)$ for compactly supported potentials takes the form $\Id + A(k)$, where $A(k)$ has a smooth kernel. It follows that $S(k) - \Id$ is a compact operator for each $k$, and therefore the spectrum of $S(k)$ is discrete except at $1$. Moreover, for nonpositive potentials, the spectrum only accumulates `from above', that is, from the upper half plane \cite[Theorem 1.7.9]{Yafaev}.  The scattering amplitude $A(k)$ is analytic in $k$, which implies that the eigenprojections of $S(k)$ vary analytically with $k$, provided the eigenvalue stays away from $1$. 

The question we address here is whether the eigenvalues can pass through $1$, or more precisely, whether a phase shift $\delta(k)$ that tends upward to $\pi$ as $k \to k_0$ can be extended continuously up to, and even past,  $k = k_0$. 

In the case that  $V$ is spherically symmetric, the scattering
solutions take the form of a spherical harmonic times a function of
$r$, and ODE analysis establishes that the phase shifts are analytic
for \emph{all} $k > 0$.  Furthermore, Levinson's theorem for central
potentials can be used to guarantee that eigenvalues of $S(k)$ do pass through $1$.

\medskip

\noindent \textbf{Levinson's theorem (central potentials).}
 \textit{Given $V(r) \in C^{\infty}_c(\mathbb{R})$ and a spherical harmonic $\phi$, let $\alpha(k)$ be the eigenvalue of $\phi$ for the scattering matrix $S(k)$ of $H = \Delta - V(\absv{x})$, i.e.\ $S(k)\phi = \alpha(k)\phi$.  Let $n$ be the dimension of the subspace  of $L^2$ eigenvectors of $H$ of the form $u = a(r)\phi$.  Then the counterclockwise winding number of $\alpha(k)$ as $k$ goes from $\infty$ to $0$ satisfies
  \begin{equation}
  - \frac{1}{2\pi i} \int_0^{\infty} \frac{\alpha'(k)}{\alpha(k)} dk = n + \nu,
  \end{equation}
  where $\nu = 1/2$ if $\phi \equiv 1$ and there is a half-bound state, and is $0$ otherwise.
}

\medskip

\noindent See Theorem XI.59 of \cite{RSIII} for a proof of this well-known fact in the case $\phi \equiv 1$, and equation (5.15) of \cite{N1960} for the general case.  For a fixed, non-positive, spherically symmetric $V$, the number of bound states of $\Delta - \lambda V$ with angular part $\phi$ grows at least as fast as $c \sqrt{\lambda}$ for some constant $c > 0$ (Theorem XIII.9 of \cite{RSIIII}), so for such $\lambda$ there are eigenvalues of the scattering matrix which pass continuously through $1$.

By contrast, our main result is that, for an arbitrary $C_c^\infty$ potential function, phase shifts \textit{cannot} necessarily be so continued: we give an explicit example of a potential well for which 
\begin{itemize}
\item 
the scattering matrix $S(k)$ does not have $1$ as an eigenvalue for any $k > 0$, and 
\item
there exists $k_0 > 0$ and an eigenvalue branch $e^{2i\delta(k)}$ of $S(k)$ such that $\delta(k) \uparrow \pi$ as $k \downarrow k_0$.
\end{itemize}
Lest this seem bizarre, we mention that the corresponding phenomenon for obstacle scattering has already been observed by Eckmann-Pillet \cite{EP1995}. 
In both cases, the source of the phenomenon is the same: for a compactly supported perturbation of $\RR^n$, if $1$ is an eigenvalue of $S(k)$ then there is a generalized eigenfunction that outside the perturbation agrees \emph{exactly} with a generalized eigenfunction of the free Laplacian (which has scattering matrix $S(k)$ equal to the identity, and hence has every eigenvalue $1$).  
See Section \ref{sec:eval1} for an elementary (and well-known) proof of this fact. This is an extremely restrictive situation, and for certain perturbations one can show that it is not possible. 

Our interest in the continuity of phase shifts through multiples of
$\pi$, or equivalently in eigenvalues of $S(k)$ through $1$, came from
reading a 1989 paper of R. Newton published in \emph{Annals of
  Physics} in which, in particular, it is claimed that one can label
the phase shifts of the scattering matrix for any $C_c^\infty$
potential in $\RR^3$ (actually, Newton considers the larger class of
bounded potentials with exponential decay) in such a way that they are
continuous functions of $k \in (0, \infty)$. This is then used to
claim a `micro-Levinson theorem' relating phase shifts at $k=0$ to the
non-positive spectrum of $H$.  But it is straightforward to show that
for any nonnegative $V \in C_c^\infty(\RR^3)$ and $\lambda$
sufficiently large, some phase shift of $- \lambda V$ will approach $\pi$ from below
as $k$ approaches some finite positive value $k_0$ from above (see
Section \ref{sec:monotone}), and  so Newton's result would imply that the scattering matrix
of $-\lambda V$ at energy $k_{0}$ has $1$ as an eigenvalue.  Our example shows that this is not true, and therefore shows that Newton's claimed theorem is incorrect. Further discussion about Newton's paper is given in Section~\ref{sec:newton} below. 

The scattering matrix for nonpositive potentials is not the only family of operators of interest in spectral analysis in which there is a one-sided accumulation of spectrum and where one is interested in eigenvalues approaching the accumulation point. Another example is the Neumann-to-Dirichlet operator $N(k)$ for a smooth bounded domain $\Omega$ in $\RR^n$ (or more generally a Riemannian manifold with boundary). This operator, defined for complex $k$ (except for $k^2$ in the Neumann spectrum of $\Omega$), takes 
$L^2(\dOmega)$ to $L^2(\dOmega)$ and maps $f \in C^\infty(\dOmega)$ to the boundary value of the function $u \in C^\infty(\Omega)$ satisfying the Helmholtz equation $(\Delta - k^2) u = 0$ and the Neumann boundary condition $d_n u = f$. This operator is a pseudodifferential operator of order $-1$ with positive principal symbol, and therefore has an accumulation of eigenvalues at $0$ from above. Eigenvalues of $N(k)$ are monotone increasing in $k$  \cite{Friedlander} and, for every Dirichlet eigenvalue $k_0^2$, there is an eigenvalue $\beta(k)$ of $N(k)$ such that $\beta(k) \uparrow 0$ as  $k \uparrow k_{0}$. The Neumann-to-Dirichlet operator is quite closely analogous to the scattering matrix (and the analogue becomes even closer if one considers the Cayley transform of  $N(k)$, which is a family of unitary operators defined for \emph{every} $k > 0$ and depending analytically on $k$). But there is an important difference between the two cases: in the case of $N(k)$, the eigenvalue branches $\beta(k)$ tending to zero as $k \uparrow k_0$ always have a continuous extension to $k_0$, with the eigenfunction at $k_0$ being the normal derivative of the corresponding Dirichlet eigenfunction on $\Omega$. This can be traced to the fact that the   eigenfunction branch corresponding to $\beta(k)$ has a weak limit in $H^1(\Omega)$ and thus, thanks to the compact embedding $H^1(\Omega) \to L^2(\Omega)$, a \emph{strong} limit as $k \uparrow k_0$, which is necessarily nonzero \cite{Barnett-Hassell}. By contrast, the generalized eigenfunction branch for the scattering problem $(H - k^2) u = 0$ corresponding to a phase shift $\delta(k)$ may have only a weak limit --- which may be zero --- if $\delta \uparrow \pi$ as $k \downarrow k_0$. 

\medskip

The authors would like to thank Lennie Friedlander, Rafe Mazzeo, and Yuri Safarov for helpful conversations during the writing of this paper.

\section{Consequences of the scattering matrix having eigenvalue $1$}\label{sec:eval1}

Suppose that we have a compactly supported perturbation, $H$, of the Laplacian $\Delta$ on $\RR^n$, such that the scattering matrix $S(k)$, $k > 0$, has an eigenvalue equal to $1$. Notice that since $S(k) = \Id + A(k)$, where $A(k)$ has a smooth kernel, the eigenfunction, say $q(\omega)$, is smooth. So there is a generalized eigenfunction $u$ of $H$ having the asymptotics 
$$
u = r^{-(n-1)/2} \Big( e^{-ikr} q(\omega) + e^{ikr} e^{i\pi(n-1)/2} q(-\omega) \Big) + O(r^{-(n+1)/2}). 
$$
Since the scattering matrix for the zero potential is the identity operator for all $k$, there is also a free generalized eigenfunction $u_f$, satisfying $(\Delta - k^2) u_f = 0$ on $\RR^n$, satisfying  
$$
u_f = r^{-(n-1)/2} \Big( e^{-ikr} q(\omega) + e^{ikr} e^{i\pi(n-1)/2} q(-\omega) \Big) + O(r^{-(n+1)/2}). 
$$
It follows that $u - u_f = O(r^{-(n+1)/2})$ near infinity. If outside some large ball in $\RR^n$ we expand $u - u_f$ in spherical harmonics, then we find that the coefficients are functions $j_l(r)$ such that $r^{(n-2)/2} j_l(r)$ satisfy Bessel's equation of order $l + (n-2)/2$, and are $O(r^{-3/2})$ as $r \to \infty$. As the only such solutions are identically zero, we find that $ u = u_f$ outside any ball containing the perturbation. Then applying standard unique continuation theorems, such as \cite[Theorem 17.2.6]{Hor}, we find that $u = u_f$ outside the support of the perturbation; in the case of potential scattering, this means outside the support of $V$. 

It is now straightforward to understand the observation of Eckmann-Pillet that the typical smooth obstacle $\Omega$ in $\RR^n$, endowed with Dirichlet boundary conditions, will never have $1$ as an eigenvalue of its scattering matrix $S(k)$ for any $k > 0$. For if a free plane wave $u_f$ agrees with a generalized eigenfunction for the exterior domain $\RR^n \setminus \Omega$, then $u_f$ vanishes on $\dOmega$. But $u_f$ is real analytic, so that would imply that $\dOmega$ is contained in the zero set of a nontrivial real analytic function, and of course this is not true for a generic smooth obstacle. (As an example, take any smooth compact obstacle whose boundary contains an open subset of a hyperplane.) 

On the other hand, the main result of Eckmann
and Pillet was that for each $k_0$ such that $k_0^2$ is a Dirichlet eigenvalue of $\Omega$, there is an eigenvalue branch $\beta(k)$ that tends to $1$ as $k \uparrow k_0$. This furnishes many examples of cases where an eigenvalue branch tending to $1$ does not extend continuously to include the value $1$, i.e.\ when the accumulation point of the spectrum is reached, the eigenfunction ceases to exist.   See \cite{EP1995} for the details.

We introduce some terminology to describe these situations. We shall
say that a potential function $-V$ is \textbf{partially transparent}
at frequency $k_0 > 0$ if its scattering matrix $S(k_0)$ has an
eigenvalue $1$; \textbf{almost partially transparent} at frequency $k_0
\geq 0$ if there is an eigenvalue branch $e^{2i\delta(k)}$ for $k >
k_0$ that tends to $1$ as $k \downarrow k_0$; and \textbf{completely
  non-transparent} if the scattering matrix is not partially
transparent for any $k$, i.e.\ if the scattering matrix has no
eigenvalue equal to $1$ for any $k > 0$.  (Note that partial
transparency implies almost partial transparency.)

In Section \ref{sec:potential}, we give an example of a completely non-transparent potential.  The example is not in the least bit pathological: it is simply a double well potential, where each well  is smooth, compactly supported and spherically symmetric,  
and the only subtlety is that the ratio between the radii of these wells is required to avoid a countable set of values (which may be dense). We believe that this property of being completely non-transparent   is generic for potential wells in $C_c^\infty(\RR^n)$, but we do not attempt to prove this in this note.

\section{Model Example on $\ell^{2}$} \label{sec:littleell2}

We now give an explicit example of a family of
operators with an eigenvalue branch that vanishes at an accumulation point. This can be regarded as a model for the way in which an eigenbranch for the scattering matrix can disappear when the eigenvalue hits $1$. This example is illustrative only and is not used in the remainder of the paper.

Let $e_{j}$, $j = 1, 2, \dots$ denote the standard basis of $\ell^{2} =
\ell^{2}(\mathbb{N})$, and let $z_{0}$ be any vector with
\begin{equation}
  \label{eq:z0def}
  \begin{split}
    \langle z_{0}, e_{j} \rangle &> 0 \mbox{ for all } j \\
    \langle z_{0}, e_{j} \rangle &= \langle z_{0}, e_{k} \rangle \iff
    j = k \\
    \norm{ z_{0} }&\leq 1.
  \end{split}
\end{equation}
For example, we can take $z_0 = \frac1{2} \sum_j e_j/j$. 

We define a self-adjoint compact operator $T$ on $\ell^{2}$ by
\begin{equation}
  \label{eq:T0def}
  T(e_{j}) := \absv{\langle z_{0}, e_{j} \rangle}^{3} e_{j},
\end{equation}
and perturb it by a family of rank one self-adjoint operators parametrized by $k
\in \mathbb{R}$, as follows:
\begin{equation}
  \label{eq:T0def}
  T_{k}(z) := T(z) + k \langle z_{0}, z \rangle z_{0}.
\end{equation}
Notice that $T$, and hence $T_k$, has spectrum accumulating at $0$, from above only. 

We will show that $T_{k}$ has the following properties
\begin{enumerate}
\item[(i)] for $k \geq 0$, $T_k$ has only positive eigenvalues;
\item[(ii)] for $k < 0$, there is exactly one negative eigenvalue $\alpha(k)$ of
  $T_{k}$, satisfying $k < \alpha(k) < 0$;
\item[(iii)] $T_{k}$ shares no eigenvalue in common with $T_{0} = T$ for $k
  \neq 0$, and for all $k$, every eigenvalue of $T_k$ is simple.  
\end{enumerate}
The upshot is that for $k < 0$ there is an eigenvalue $\alpha(k)$ of $T_{k}$
which approaches $0$ from below as $k \uparrow 0$, but  there is no corresponding zero eigenvalue at $k = 0$.   By  point (iii), there is no possible continuation of
$\alpha(k)$ to $k \geq 0$, even allowing a removable singularity at $k=0$:  the eigenvalue branch $\alpha(k)$ simply ceases
to exist.

The proofs of these three facts are elementary.  For $k \geq 0$, the
operator $T_{k}$ is manifestly
positive, and for $k < 0$, it is manifestly bounded below by $k$.  For any $k$, $T_{k}$ restricted to $\spn
\langle z_{0} \rangle^{\perp}$ is equal to $T$, hence it is positive
off a codimension $1$ subspace.  On the other hand\begin{equation*}
  \langle T_{k}(e_{j}), e_{j} \rangle =   \absv{\langle z_{0},e_{j}
    \rangle}^{3} + k \absv{\langle z_{0},e_{j} \rangle}^{2},
\end{equation*}
so for $k < 0$ this is negative for large $j$.  Thus there is exactly one
negative eigenvalue if $k < 0$.

To prove (iii),  suppose that $\lambda$ is an eigenvalue of
both $T$ and $T_{k}$ for $k \neq 0$.  All eigenspaces of $T$ are
simple, so $\lambda = \absv{\langle z_{0}, e_{j} \rangle}^{3}$ for
some $j$ and the eigenvector is $e_{j}$. Let $z$ be the eigenvector of $T_k$. Then we have 
$$
\lambda \langle z, e_j \rangle = \langle T_k z, e_j \rangle = \langle z, T_0 e_j \rangle 
\implies \langle (T_k - T_0) z, e_j \rangle  = 0.
$$
This means that 
$$
k \langle z_0, z \rangle \langle z_0, e_j \rangle = 0.
$$
Since $\langle z_0, e_j \rangle \neq 0$ we find that 
\begin{equation}
\langle z_0, z \rangle = 0. 
\label{ipz}\end{equation}
But this implies that $T_k z = T_0 z$. Therefore, $z$ is an 
eigenfunction of $T_0$, i.e.\ $z = e_i$ for some $i$. But this contradicts \eqref{ipz}, since $\langle z_0, e_i \rangle \neq 0$. 
Finally we show that $T_k$ has only simple eigenvalues. This is true by construction for $k = 0$. For $k \neq 0$, if $T_k$ had an eigenspace of two or more dimensions then it would contain a nontrivial eigenvector $w$ orthogonal to $z_0$. But then, as before, we would have $T_k w = T_0 w$ so $w$ would be an eigenvector of $T_0$, contradicting the fact just proved that $T_k$ and $T_0$ have no eigenvalues in common. 

The spectrum of $T_k$ as a function of $k$ therefore is as in Figure \ref{fig:model}, with necessarily an infinity of avoided crossings due to property (iii) above. 
\begin{figure}
  \centering
  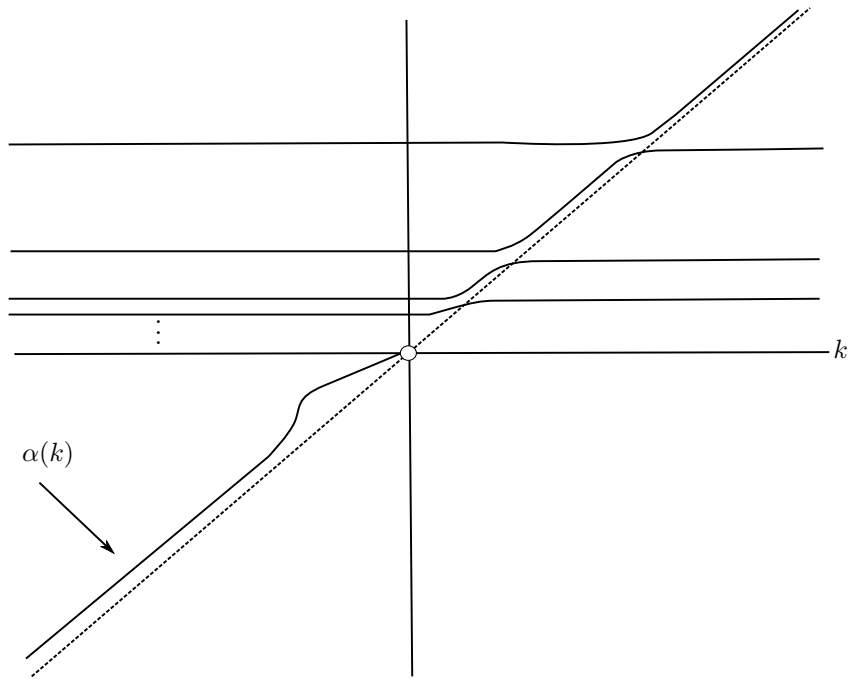
  \caption{Schematic depiction of the spectrum of $T(k)$. The negative eigenvalue branch, $\alpha(k)$, vanishes as $k = 0$, the other eigenvalues avoid crossing, and the largest eventually increases as $\| z_0 \|^2k + O(1)$. The dotted line is the line $y = \| z_0 \|^2 k$.}
  \label{fig:model}
\end{figure}

\section{Completely non-transparent potential}\label{sec:potential}

We now give an example of a compactly supported, smooth potential $V >
0$  which is completely non-transparent.  In fact we show a slightly stronger result; for any sequence of `strengths', i.e.\ positive real numbers $\lambda_i \to \infty$, we find a $V$ such that all of the potentials $-{\lambda}_{i}V$ are completely non-transparent.

Given any $W \in C_{c}^{\infty}(\mathbb{R}^{n})$, let $S_W(k)$ denote
the scattering matrix of $\Delta -  W$. If $ W$ is partially transparent at frequency $k$, then, as discussed in Section \ref{sec:eval1}, there are solutions
\begin{equation}\label{eq:wschroed}
  \begin{split}
    \lp \Delta -  W - k^{2} \rp u &= 0 \\
    \lp \Delta  - k^{2} \rp u_{f} &= 0 \\
  \end{split}
\end{equation}
such that  $u =
  u_{f}$ on the complement of $\supp W$. Let $\Omega$ be any smooth bounded domain containing the interior of $\supp W.$  Then, at the boundary of $\Omega$,
  \begin{equation}\label{eq:doubledton}
  \begin{split}
      u \rvert_{\p \Omega} &= u_{f} \rvert_{\p \Omega} \\
    \p_{\nu} u \rvert_{\p \Omega} &= \p_{\nu}  u_{f} \rvert_{\p \Omega}.
  \end{split}
  \end{equation}
We will construct a domain so that for
any functions $u$ and $u_{f}$ satisfying \eqref{eq:wschroed}
\textbf{on $\Omega$}, equation \eqref{eq:doubledton} holds only
when $u  \equiv u_{f} \equiv 0$.

% \subsection{D-to-N map}
% Given $V$ and ${\lambda}$, this problem can be rephrased in terms of the
% Dirichlet-to-Neumann map.  Let $\Theta_{f}(k)$ and
% $\Theta_{{\lambda}}(k)$ denote the Dirichlet-to-Neumann maps for $\Delta
% - k^{2}$ and $\Delta - {\lambda}^{2}V - k^{2}$, respectively, on $\Omega$.  If either
% of these operators has a kernel, $\Theta$ is defined by sending any $f
% \in C^{\infty}(\p \Omega)$ to the normal derivative of the unique
% solution that is $L^{2}$ orthogonal to the kernel.  

% For $u$ and $u_{f}$ described above, we 

Let $\chi \in C^{\infty}_c(\mathbb{R})$ denote be a bump function with $\chi \geq 0$, $\chi(r) \equiv 1$ for $r < 1/4$, and
$\supp \chi \in \set{r \leq 1/2}$.  Also choose $x_0 \in \RR^n$ with $|x_0| > 1$ and $R > 0$ such that  the intersection $\overline{B_0(1)} \cap \overline{B_{x_0}(R)}$ is empty, that is, so that $R < |x_0| -1$. We define a potential
\begin{equation}
  \label{eq:2}
  \fbox{$ V_{R}(x) := \chi(\absv{x}) + R^{-2} \chi(\absv{x - x_0}/R) $.}
\end{equation}
We think of $x_0$ as fixed throughout.
\begin{theorem}\label{thm:notransparency}
  Let $\Omega_{R} = B_{0}(1) \cup B_{x_0}(R)$ where $|x_0| > 1$. Then for  any sequence of
  positive numbers ${\lambda}_{i} \to \infty$, there is a countable set
  $\Lambda$ so that for $R \not \in \Lambda$ and $0 < R < |x_0| - 1$, there are
  no non-zero simultaneous solutions to \eqref{eq:wschroed} and
  \eqref{eq:doubledton} with $W = \lambda_i V_{R}$.  Thus, for each  $\lambda_i$, the scattering matrix for $S_{\lambda_{i} V_{R}}(k)$ does not have $1$ as an eigenvalue for any $k > 0$, i.e.\ $-\lambda_i V_R$ is completely non-transparent.
\end{theorem}

Before we begin the proof, we discuss the case of a single well potential 
$$W_{R}(\absv{x}) = \chi(x/R)/R^2.$$  If we
 label the spherical harmonics in the standard way, $\phi_{lm}$, where
 $\Delta_{S^{n-1}}\phi_{lm} = l(l + n - 2) \phi_{lm}$, any
solution $u, u_{f}$ to \eqref{eq:wschroed} on $B_{0}(R)$ can be
written
\begin{align}
  \label{eq:things}
  u_{f}(r\omega) &= \sum_{ |m| \leq l, l=0}^\infty a_{lm} r^{-(n-2)/2}
   J_{l + (n-2)/2}(kr) \phi_{lm}(\omega) 
   \end{align}
where $J_{\nu}$ is the standard Bessel function of order $\nu$, and
\begin{align}
  \label{eq:things}
  u(r\omega) &= \sum_{ |m| \leq l, l=0}^\infty b_{lm} r^{-(n-2)/2} \JJ_{l, k, \lambda, R}(r) \phi_{lm}(\omega)
  \end{align}
where $\JJ_{l, k, \lambda, R}(r)$ is the unique solution to
\begin{equation}
  \label{eq:almostbessell}
  \lp - \p_{r}^{2} - \frac{1}{r} \p_{r} + \frac{(l+(n-2)/2)^2}{r^{2}} -
  \lambda W_{R}(r) - k^{2} \rp \JJ_{l, k, \lambda, R}(r) = 0,
\end{equation}
which is equal to $ J_{l + (n-2)/2}(\sqrt{\lambda/R^{2} + k^{2}} \,  r)$ near $r =
0$.  (This is, up to scale, the unique regular solution, since $W_{R} = 1/R^2$ near $r = 0$.)

From this it is clear that a necessary and sufficient condition for solving both
\eqref{eq:wschroed} and \eqref{eq:doubledton} is that there exist a non-negative integer $l$ and a $k > 0$ such that the Wronskian
\begin{equation}
  \begin{split}
    D_{l, k, \lambda, R}(r) := 
    \det \lp
      \begin{array}{cc}
      J_{l + (n-2)/2}(kr) & \JJ_{l, k, \lambda, R}(r) \\ 
     \p_r \big(J_{l + (n-2)/2}(kr) \big) &  \p_{r} \big(\JJ_{l, k, \lambda, R}(r) \big)
  \end{array}
  \rp
  \end{split}
\end{equation}
satisfies
\begin{equation}
  \label{eq:bessellzeroesR}
  D_{l, k, \lambda, R}(R) = 0.
\end{equation}
We will prove 
\begin{lemma}\label{thm:bessellzeroeslemma}
  For fixed ${l}$, ${\lambda}$, and $R$, the set of zeroes of $D_{{l}, k,
  {\lambda}, R}(R)$ as a function of $k$ is discrete (hence countable) in
  $(0, \infty)$.  For fixed ${l}$, ${\lambda}$, and $k$, the set of zeroes of $D_{{l}, k,
  {\lambda}, R}(R)$ as a function of $R$ is discrete (hence countable) in
  $(0, \infty)$.
\end{lemma}

Assuming the lemma, the theorem follows easily.

\begin{proof}[Proof of Theorem \ref{thm:notransparency}]
  Fix a sequence $\lambda_{i} \to \infty$, and a particular element
  $\lambda$, thereof.  By the proceeding discussion, given any $R < |x_0| - 1$, there is a solution
  to \eqref{eq:wschroed} and \eqref{eq:doubledton} on $\Omega_{R}$ if
  and only if there are $l, \wt{l} \geq 0$, and $k > 0$ so that
  we have both
\begin{equation}
  \label{eq:bessellzeroessystems}
  \begin{split}
    D_{l, k , {\lambda}, R}(R) &= 0 \\
    D_{\wt{l}, k , {\lambda}, 1}(1) &= 0 .
    \end{split}
\end{equation}
Fixing $l, \wt{l}$, consider the set of $R > 0$ for which there is a solution to \eqref{eq:bessellzeroessystems}.  Using the first part of Lemma~\ref{thm:bessellzeroeslemma}, there
are only countably many solutions, $k_{i}$, to the second equation in \eqref{eq:bessellzeroessystems}.  For each $k_{i}$, using the second part of Lemma~\ref{thm:bessellzeroeslemma}, there are only countably many
solutions $R_{i,j}$ to $    D_{l, k_{i} , {\lambda}, R_{i,j}}(R_{i,j}) = 0$, and thus there
are only countably many $R$ for which the system \eqref{eq:bessellzeroessystems}
admits a solution.  There are
only countably many pairs $(l, \wt{l})$ of nonnegative integers,
so the set of $R$ such that \eqref{eq:bessellzeroessystems} holds for any pair 
$(l, \wt{l})$ is countable.

Thus, for each $\lambda_{i}$, there are no solutions to
\eqref{eq:bessellzeroessystems} with $l, \wt{l} \in \{ 0, 1, 2, \dots \}$ and $R < |x_0| -1$ not in some
countable set $\Lambda_{i}$.  Setting $\Lambda = \bigcup_{i}
\Lambda_{i}$ proves the theorem.
\end{proof}

It remains to prove the lemma.

\begin{proof}[Proof of Lemma \ref{thm:bessellzeroeslemma}]
To simplify notation, for fixed $R$, ${l}$, and ${\lambda}$ 
set $F(k) := D_{l, k, \lambda, R}(R)$, and for fixed $k$, ${l}$, and
${\lambda}$, set $G(R) =  D_{l, k, \lambda, R}(R)$.  Our goal is to
prove that both $F(k)$ and $G(R)$ have countably many zeroes.

To see this, note that by scalling the $r$ variable in \eqref{eq:almostbessell}, $\JJ_{l, k, \lambda, R}(r) = f(l, kR, \lambda, r/R)$, where  
$f$ satisfies the differential equation 
\begin{equation}
  \label{eq:almostbessell2}
  \lp - \p_{r}^{2} - \frac{1}{r} \p_{r} + \frac{(l+(n-2)/2)^2}{r^{2}} -
  \lambda W_{1}(r) - k^{2}  \rp  f(l, k, \lambda, r) = 0. 
\end{equation}
The two terms in $D_{l, k, \lambda, R}(R)$ involving $\JJ_{l, k, \lambda, R}$ are,
$\JJ_{l, k, \lambda, R}(R) = f(l, kR, \lambda, 1)$ and \linebreak $\p_{r} \lp
\JJ_{l, k, \lambda, R}(r) \rp \rvert_{r = R} = \p_{r} f(l, kR, \lambda, 1) / R$.  These are analytic functions of $kR$ by Theorem XI.56 in
\cite{RSIII}.  Bessel functions are analytic, so $D_{l, k, \lambda, R}(R)$ is analytic in both $k$ and $R$.

It therefore suffices to check that neither $F(k)$ nor $G(R)$ is
identically zero, which we do using the Sturm comparison theorem and taking
$k$ and $R$ large in comparison with $\lambda$.  Note that $D_{l, k, \lambda, R}(R)$ is nonzero precisely when $\JJ_{l, k, \lambda, R}$ is not a multiple of $J_{l+(n-2)/2}$ for $r \geq R/2$. We show this is true for when $kR$ is large by comparing $\JJ_{l, k, \lambda, R}$ to the solutions to equation \eqref{eq:almostbessell} when the potential $-\lambda W_R$ is replaced by either $-\lambda R^{-2} 1_{[0, R/4]}$ or $-\lambda R^{-2} 1_{[0, R/2]}$. Since $R^{-2} 1_{[0, R/4]} \leq W_R \leq R^{-2} 1_{[0, R/2]}$, the Sturm comparison theorem tells us that the $N$th zero of $\JJ_{l, k, \lambda, R}$ is between that of the solution to equation \eqref{eq:almostbessell} with potential $-\lambda W_R$ replaced with the characteristic functions above. Since the solutions for these are $J_{l+(n-2)/2}(\sqrt{k^2 + \lambda R^{-2}} \, r)$ on the support of the characteristic function, and  since $\sqrt{k^2 + \lambda R^{-2}} = k + O(1/(kR^2))$, we find that the zeroes of $\JJ_{l, k, \lambda, R}(r)$ for $r \geq R$, which are $\sim \pi/k$ apart, are shifted by an amount between $c_1 (kR)^{-2}$ and $c_2 (kR)^{-2}$ relative to those of $J_{l + (n-2)/2}(kr)$, where $0 < c_1 < c_2$. This shows that when $k$ is large for fixed $R$, or for $R$ large for fixed $k$, then $\JJ_{l, k, \lambda, R}(r)$ is not equal to $J_{l + (n-2)/2}(kr)$ for $r \geq R$ and hence that $D_{l, k, \lambda, R}(R) \neq 0$.

\begin{comment}
A direct computation shows that 
\begin{equation*}
  \p_{r} r D_{l, k, \lambda, R}(r) = - \lambda^{2}rW_{R}(r)  J_{l + (n-2)/2}(kr) \JJ_{l, k, \lambda, R}(r),
\end{equation*}
so, since $\lim_{r \to 0} r D_{l, k, \lambda, R}(r) = 0$,
\begin{equation}\label{eq:fR}
  D_{l, k, \lambda, R}(R) = - \frac{1}{R} \int_{0}^{R} s \lambda^{2} W_{R}(s)   J_{l + (n-2)/2}(ks) \JJ_{l, k, \lambda, R}(s) ds
\end{equation}
For $f$ as in \eqref{eq:almostbessell2}, and changing variables to $ks = s'$, we have
\begin{equation}\label{eq:fR2}
  k R^{2} D_{l, k, \lambda, R}(R) = - \frac{1}{kR} \int_{0}^{kR} s'
  \lambda^{2} W_{kR}(s')   J_{l + (n-2)/2}(s) f(l, \lambda, s/kR, kR) ds
\end{equation}
By \eqref{eq:almostbessell2}, we expect $f(l, \lambda, s/kR, kR)$ to behave like $ J_{l + (n-2)/2}(s)$ for large $kR$.  Approximating using variation of parameters gives
\begin{equation*}
 k R^{2} D_{l, k, \lambda, R}(R)   =  - \frac{1}{kR} \int_{0}^{kR} s'
  \lambda^{2} W_{kR}(s')   J^2_{l + (n-2)/2}(s) ds + \mathcal{O}(1/\sqrt{kR}),
\end{equation*}
Using  the standard expansions for Bessel
functions near infinity it is easy to see that the integral converges to a non-zero constant for $kR$ large.  Thus $ k R^{2} D_{l, k, \lambda, R}(R) $ is nonzero for large $k$ with $R$ fixed, and for large $R$ with fixed $k$.  This proves the lemma.
\end{comment}

\end{proof}

\section{Monotonicity of phase shifts and the existence of almost partially transparent frequencies}\label{sec:monotone}

For any
potential $V \geq 0$, it is not difficult to show, using the monotonicity result of \cite{HR1976}, that the number of non-zero almost partially transparent
frequencies of $\Delta - \lambda V$ is unbounded as $\lambda
\to \infty$.  We sketch an argument now.  Without loss of generality, assume that
$V(0) \neq 0$, and let $\chi = \chi(r) \in C_{c}^{\infty}$ be a smooth, non-negative, spherically-symmetric 
function with $V \geq \chi$.  Setting $V(s) := s V+ (1 - s) \chi$,
let $S_{s}(k)$ be the scattering matrix for $\Delta - \lambda V(s)$ at
  frequency $k$.  Let $\alpha(s,k) = \exp(2i\delta(s,k))$ be an eigenvalue of $S_{s}(k)$.  As long as it is not equal to $1$, $\alpha(s,k)$ can be taken analytic in $s$.  If $\alpha(s,k) \neq 1$ for all $s \in [0,1]$, Theorem 1 of \cite{HR1976} gives
  \begin{equation}\label{eq:mono}
    \frac{ \p \delta}{\p s} \geq 0 \mbox{ for } s \in [0,1]. %, \mbox{ so in particular } \log(\alpha(1)) > \log(\alpha(0)).
  \end{equation}
%(Here the logarithm takes values e.g.\ in $[0,2\pi)$.)  
Recall that for $s = 0$, i.e.\ for the potential $-\lambda \chi$, the
scattering matrix is diagonal with respect to (say, the standard)
basis of spherical
harmonics, $\phi_{lm}$.  The eigenvalues $\alpha_{lm}(k) = \exp(2i\delta_{lm}(k))$ are
defined continuously for all $k$, and can be taken so that
$\delta_{lm}(k) \to 0$ as $k \to \infty$.  As we saw in our discussion of Levinson's theorem for central potentials in the introduction, by taking $\lambda$ sufficiently large,  the counterclockwise winding number of $\alpha_{lm}(k)$ can be assumed bigger than $1$.   For $k$ taken large so that $S_{s}(k)$ is
very close to the identity, let $\alpha(s,k)$ be an eigenvalue with
$\alpha(0,k) = \alpha_{lm}(k)$.  By \eqref{eq:mono}, $\delta_{lm}(k) < \delta(1,k)$.
%(The logarithms are well-defined for $k$ large.)  
If $\lambda$ was chosen
large enough so that
$\delta_{lm}(k)$ crosses $\pi$, say at $k_{0}$, then $\delta(1,k)$ must approach
$\pi$ from below at some frequency $k' \geq k_{0}$ --- see Figure~\ref{fig:almosttrans}.  This produces a non-zero almost partially transparent eigenvalue for each pair $(l,m)$ for which the winding number of $\alpha_{lm}(k)$ is bigger than $1$.

\begin{figure}[htbp]
\begin{center}
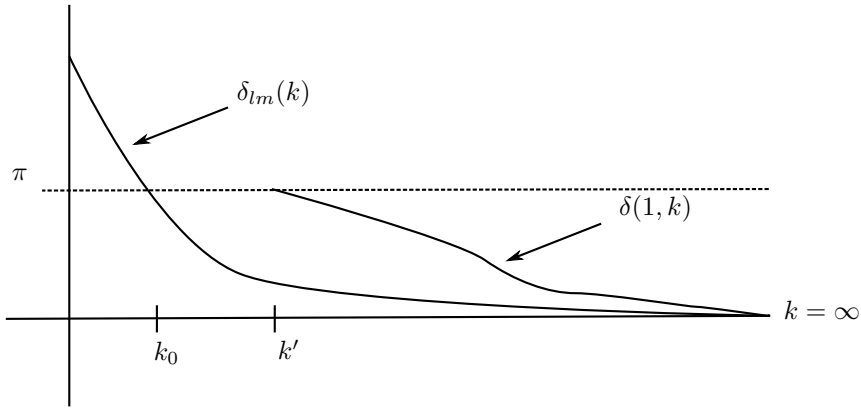
\caption{For $k$ large, the phase shifts are monotone in $s$.}
\label{fig:almosttrans}
\end{center}
\end{figure}

As a simple corollary to this argument we have
\begin{corollary}\label{thm:onlyalmost}
For $V_R$ and $\lambda_i \to \infty$ as in Theorem \ref{thm:notransparency}, and for $\lambda_i$ large enough, there exist eigenvalues of the scattering matrix $S_{\lambda_i V_R}$ which approach $1$ at non-zero frequencies.  Since $- \lambda_i
 V_R$ is completely non-transparent, these limiting frequencies are almost partially transparent but not partially transparent.
 \end{corollary}

\section{On Newton's `micro Levinson theorem' for non-central potentials}\label{sec:newton}

In 1989, R. Newton published a paper \emph{The spectrum of the Schr\"odinger $S$ matrix: low energies and a new Levinson theorem} \cite{N1989}, claiming the following result: 

\medskip

\noindent \textbf{Claimed Theorem.} \textit{Assume that the
potential [is smooth and exponentially decaying].  Then each
eigenphase shift $\delta_{lm}(k)$ may be defined to be a continuous
function of $k$, to vanish at $k \to \infty$, and so that its value at the
origin is $\delta_{lm}(0) = \pi (\mathcal{N}_{lm} + \nu)$ where
$\mathcal{N}_{lm}$ is the number of bound states associated with the
pair $(l,m)$, $\nu = 1/2$ if $l = 0$ and there is a half-bound state,
and $\nu = 0$ otherwise.
}

\medskip

To label bound states, Newton introduces a strength parameter $\lambda$ and considers the negative spectrum for the family $\lambda V$; as $\lambda \to 0$ the bound states approach zero energy and then disappear, and the label is related to the asymptotic spatial behaviour of the limiting zero energy solution. 

Corollary \ref{thm:onlyalmost} above shows that Newton's claimed theorem is
incorrect. Indeed, for $V_R$ and $\lambda_i$ as in the corollary, there are eigenvalues of the scattering matrix of $\Delta - \lambda_i V$ which approach $1$ as $k \downarrow k_0$ for some non-zero frequency $k_0$, but such that $1$ is not an eigenvalue of $S(k_0)$.  Therefore, the phase shifts of $S(k)$ cannot always be taken continuous on $k \in (0, \infty)$. 

We mention some other problems with Newton's paper. 
\begin{itemize}

\item The proposed labelling of bound states is flawed. It seems to depend implicitly on the assumption that the limiting eigenfunctions of the scattering matrix $S_{\lambda V}(k)$ as $k \to 0$ is independent of $\lambda$, but this is not the case.

\item Lemma 4.2 of \cite{N1989} is incorrect. Using the notation in the lemma, if $u$ is a zero energy bound state of $(\Delta + V) u = 0$ with angular dependence $r^{-l-1} \mathcal{Y}_{ln}(\omega)$, where $\mathcal{Y}_{ln}(\omega)$ is a spherical harmonic with angular momentum quantum number $l$ , then $u = Ku$ and $P_{ln} u = u$, and hence $u = K P_{ln} u$. However, the converse is certainly not true: from $u = K P_{ln} u$ we are not able to deduce that 
both $u = Ku$ and $u = P_{ln} u$, which would be required to conclude Lemma 4.2. 
\end{itemize}

The second point is of some significance: if Lemma 4.2 were correct for every $C_c^\infty(\RR^3)$ potential $V$, there would be infinitely many $\lambda$ such that $\Delta - \lambda V$ had a zero eigenvalue with eigenfunction behaving as $r^{-l-1} \mathcal{Y}_{ln}(\omega)$ at infinity, for some spherical harmonic 
$\mathcal{Y}_{ln}(\omega)$ (depending on $\lambda$) with angular momentum quantum number $l$. However, we are sure this is not the case. In fact, we conjecture that the  generic potential $V \in C_c^\infty(\RR^n)$ is such that all the zero energy eigenfunctions of $\Delta + \lambda V$ are half-bound states.

\section{Open problems}

This short note suggests some interesting open problems concerning potential scattering. 

\begin{itemize}
\item Is the generic potential well in $C_c^\infty(\RR^n)$ completely non-transparent? 

\item Is it true that for a generic potential $V$ in $C_c^\infty(\RR^n)$, all the zero-energy, decaying solutions $u$ to $(\Delta - \lambda V)u = 0$ are half-bound states? That is, are $L^2$ zero-energy eigenfunctions nongeneric relative to half-bound states? 

\item Is there a Levinson-type theorem for non-central potentials where instead of looking at the value of phase-shifts at $k=0$, we count the number of almost partially transparent energies (counted with multiplicity)?

\item Can one characterize, in some spectral-geometric way, the almost partially transparent frequencies of a potential well $-V \in C_c^\infty(\RR^n)$, by analogy with that for obstacle scattering \cite{EP1995}? 
\end{itemize}

\bibliographystyle{abbrv} 
\bibliography{transparency}

%\printbibliography

\end{document}